\documentclass[a4paper]{amsart}
\usepackage[utf8]{inputenc}
\usepackage{amsmath,amssymb,amsthm}
\usepackage[british]{babel}
\usepackage{eucal}
\usepackage{graphicx}
\usepackage{color}
\usepackage{url}
\usepackage{esint}
\usepackage{tikz}
\newcommand{\rot}{\operatorname{rot}}
\newcommand{\ddiv}{\operatorname{div}}

 \newtheorem{theorem}{Theorem}
 \newtheorem{lemma}[theorem]{Lemma}
 
 \newtheorem{corollary}[theorem]{Corollary}
 \newtheorem{condition}[theorem]{Condition}
 \theoremstyle{definition}
 \newtheorem{example}[theorem]{Example}
 \theoremstyle{remark}
 \newtheorem{remark}[theorem]{Remark}
\numberwithin{equation}{section}
\numberwithin{theorem}{section}

\begin{document}
%--------------------------------------------------------------------
% Title, Author, etc.
%--------------------------------------------------------------------
\title{Mixed methods and lower eigenvalue bounds}

\author{Dietmar Gallistl}
\address{Universit\"at Jena, Institut f\"ur Mathematik,
         07743 Jena, Germany}
\email{dietmar.gallistl [at symbol] uni-jena.de}
\date{}%\today
%--------------------------------------------------------------------

\begin{abstract}
It is shown how mixed finite element methods for 
symmetric positive definite eigenvalue problems related
to partial differential operators can provide guaranteed
lower eigenvalue bounds.
The method is based on a classical
compatibility condition (inclusion of kernels)
of the mixed scheme and
on local constants related to compact embeddings,
which are often known explicitly.
Applications include scalar second-order elliptic operators, linear elasticity,
and the Steklov eigenvalue problem.
\end{abstract}


\maketitle

\section{Introduction}

Variationally posed symmetric eigenvalue problems 
related to positive definite partial differential equations (PDEs)
with a compact resolvent are subject
to the Rayleigh--Ritz principle
\cite{WeinsteinStenger1972}, which characterizes the 
eigenvalues as certain minima in the corresponding
Hilbert space.
Consequently, conforming discretization methods,
which are based on subspaces of the same Hilbert space,
result in upper eigenvalue bounds.
In contrast, the identification of guaranteed
and efficient lower bounds to the eigenvalues is much
more challenging, and general principles leading to lower bounds
are not known.
In the context of finite element methods (FEMs) for some
PDE-based eigenvalue problems, major progress was achieved
by
\cite{CarstensenGedicke2014,LiuOishi2013}
where `nonstandard' methods, i.e., methods beyond 
merely conforming discretizations, were employed to prove
computable guaranteed lower bounds for the Laplacian.
These approaches were later generalized to other
eigenvalue problems in
\cite{CarstensenGallistl2014,YouXieLiu2019,GallistlOlkhovskiy2021}.
These methods have in common that they make
use of the explicit knowledge of stability or 
approximation constants for projection operators related
to the particular underlying finite element method.

By introducing a dual (or stress) variable, PDE eigenvalue
problems can be posed in an equivalent mixed formulation;
more precisely and following the terminology
of \cite{Boffi2010}, as a mixed problem of the second type.
Stability properties as well as asymptotic error estimates for 
eigenvalue problems in mixed formulation are well understood,
and the state of the art is documented in the 
review article \cite{Boffi2010} and the monograph
\cite{BoffiBrezziFortin2013}.
Mixed formulations of positive definite problems are
usually posed as saddle-point problems, which implies
a higher computational cost compared with standard discretizations.
While application-related advantages of operating with
the mechanically relevant stress variable are sometimes
mentioned for justifying and advertising mixed methods,
a decisive structural advantage for systematically using
mixed methods in eigenvalue computations has remained obscure.
It is the aim of this contribution to reveal a basic and rather
generic feature of dual mixed formulations and 
related finite element discretizations
that allows for the computation of guaranteed lower
eigenvalue bounds in many practical examples.
The involved constants are related to properties
of the underlying PDE operator rather than special properties
of the discretization space.
The applications presented in this paper include
the Laplacian, general second-order scalar coercive
operators, the Lam\'e eigenvalues of linear elasticity
(where the present method seems to be the first in the literature
to provide guaranteed lower eigenvalue bounds),
and the Steklov eigenvalue problem.
Generalizations to the Stokes or the biharmonic eigenvalue
problem are possible and briefly outlined at the end of
the paper.

The principal assumption on the discretization of the mixed
system is the inclusion of kernels, a classical compatibility
condition that for example guarantees that the equation for
the divergence holds pointwise in the case of the Laplacian.
In the usual mixed setting with discrete spaces $\Sigma_h$
and $U_h\subseteq U$ and a bilinear form $b$,
the inclusion of kernels implies that a projection $P_h$
(in many cases the orthogonal projection) from $U$ to $U_h$
exists such that for any $v\in U$,
$b(\cdot,v)$ and $b(\cdot,P_h v)$ represent the same 
linear functional over $\Sigma_h$.
A consequence proven in this paper is the commutation
property
$$
          \mathfrak P_h G = G_h P_h
$$
where $\mathfrak P_h$ is the orthogonal projection to $\Sigma_h$
with respect to a scalar product $a$
and $G$ is a gradient-like operator (from an integration
by parts formula) with its discrete analogue $G_h$.
In the simplest setting (without lower-order terms),
the first discrete eigenvalue $\lambda_{1,h}$ is the minimum of the 
Rayleigh quotient $\|G_h v_h\|_a^2/\|v_h\|_\ell^2$ for some
seminorm $\|\cdot\|_\ell$ over appropriate elements
$v_h$ from $U_h$.
Accordingly, the projection $P_h u$ of an $\ell$-normalized
first exact eigenfunction satisfies,
as a candidate for the minimum,
$$
 \lambda_{1,h} \|P_h u\|_\ell^2
 \leq 
 \|G_h P_h u\|_a^2
 \leq 
 \lambda
$$
because of the commutation property and $\|Gu\|_a^2=\lambda$.
Consequently, explicit control on the deviation
of $\|P_h u\|_\ell^2$ from $\|u\|_\ell^2=1$
results in a guaranteed lower bound for $\lambda$.
Under the assumption that there exists some $\delta_h$
such that 
$\|u-P_h u\|_\ell \leq \delta_h \|Gu\|_a$,
the following guaranteed lower bound
is established in Theorem~\ref{t:lowerbound}
\begin{equation*}
\frac{\lambda_h}{1+\delta_h^2\lambda_h}  \leq \lambda.
 \end{equation*}
In contrast to the methods proposed in
\cite{CarstensenGedicke2014,LiuOishi2013},
this paper thereby provides a methodology that covers rather general
operators in the sense that it basically requires the 
structure of saddle-point eigenvalue problems of the second type
and some control on the generic projection $P_h$
related to the inclusion-of-kernels
property,
but no particular knowledge on special interpolation or
solution operators. It therefore covers a variety of eigenvalue
problems with, e.g., variable coefficients and
lower-order terms.
A limitation of the approach as a post-processing method is
that it intrinsically is a low-order method,
which is also the case for the existing schemes
\cite{CarstensenGedicke2014,LiuOishi2013}.
The reason is that the quantity $\delta_h^2$ in the 
denominator of the lower bound is usually related to
some compact embedding and does not improve when higher-order
methods are employed.
A direct computation shows that
the difference of $\lambda$ and the lower bound cannot be
of a better order than $O(\delta^2)$.
For the Laplacian, this limitation was recently overcome
by \cite{CarstensenPuttkammer2022}, but the argument again
uses particular properties of related finite element spaces
and seems to be less universal.

\medskip
This paper is organized as follows.
Section~\ref{s:commut} proves the fundamental commutation
property for mixed methods with a compatibility
condition (inclusion of kernels).
The resulting abstract lower eigenvalue bound is 
shown in Section~\ref{s:lower}.
The subsequent sections show applications to
the Laplacian (Section~\ref{s:lapl}),
scalar elliptic operators (Section~\ref{s:elliptic}),
linear elasticity (Section~\ref{s:elast}),
and the Steklov eigenvalue problem (Section~\ref{s:steklov}).
Section~\ref{s:concl} provides
some conclusive remarks.

\section{A basic commutation property}\label{s:commut}

Let $\Sigma$, $U$ be Hilbert spaces 
(the corresponding norms are denoted by
$\|\cdot\|_\Sigma$ and $\|\cdot\|_U$)
with bounded bilinear forms
$$
 a:\Sigma\times\Sigma\to \mathbb R,
 \quad
 b:\Sigma\times U\to \mathbb R.
$$
Assume that $a$ is symmetric and positive definite
so that it induces a norm $\|\cdot\|_a=a(\cdot,\cdot)^{1/2}$,
which is in general different from the norm in $\Sigma$.
Let $\Sigma_h\subseteq\Sigma$, $U_h\subseteq U$ 
be finite-dimensional subspaces.
We remark that the finite dimensional space $\Sigma_h$ is also
a Hilbert space when endowed with the inner product $a$.
Given any $v\in U$, we write
$G_h v\in\Sigma_h$ for the solution of the system
$$
  a(G_h v,\tau_h)  = -b(\tau_h,v) 
   \quad\text{for all }\tau_h\in\Sigma_h.
$$
This system is uniquely solvable because $a$ is an inner product
on the finite-dimensional space $\Sigma_h$.
In general, the analogue problem in the infinite-dimensional
space $\Sigma$ need not be solvable because 
$\Sigma$ is not assumed complete and
$b(\cdot,v)$ is not 
assumed continuous with respect to the norm $\|\cdot\|_a$.
Let therefore $\overline\Sigma$ denote the closure of $\Sigma$ with
respect to $\|\cdot\|_a$ and let
$U_0\subseteq U$ denote the space of all elements 
$v$ of $U$ admitting
a solution $Gv\in\overline\Sigma$ to
\begin{equation}
\label{e:gradient}
  a(Gv,\tau)  = -b(\tau,v) 
   \quad\text{for all }\tau\in\Sigma.
\end{equation}
The next lemma states that the space $U_0$ is a Hilbert space
and satisfies an equivalence of norms provided
the classical inf-sup condition
\cite{BoffiBrezziFortin2013} holds.

\begin{lemma}\label{l:hilbert}
Let $(\cdot,\cdot)_U$ denote the inner product of $U$.
 The space $U_0$ with the inner product
 $$
    (v,w)_U + a(Gv,Gw) \quad\text{for any }v,w\in U_0
 $$
 is a Hilbert space.
 If the inf-sup condition 
 \begin{equation}
  \label{e:infsup}
 0<\beta=
 \inf_{v\in U \setminus \{0\}}
 \sup_{\tau\in\Sigma\setminus \{0\}}
  \frac{b(\tau,v)}{\|\tau\|_\Sigma \|v\|_U} 
 \end{equation}
 with some positive number $\beta>0$
 is satisfied,
 the norms $\|G\cdot\|_a$ and $(\|\cdot\|_U^2+\|G\cdot\|_a^2)^{1/2}$
 are equivalent on $U_0$.
\end{lemma}
\begin{proof}
 For the proof that $U_0$ is a Hilbert space, it suffices to
 show that it is complete. Let $(u_j)_j$ be a sequence in
 $U_0$ such that $u_j\to u\in U$
 with respect to $\|\cdot\|_U$ and $Gu_j\to\sigma\in\overline\Sigma$
 with respect to $\|\cdot\|_a$, for $j\to\infty$.
 Then, for any $\tau\in\Sigma$,
 $$
  a(\sigma,\tau)
  =
  a(\sigma-Gu_j,\tau) - b(\tau, u_j-u) - b(\tau,u). 
 $$
 The limit $j\to\infty$ and the continuity of $a$ and $b$
 prove
 $a(\sigma,\tau)=-b(\tau,u)$ for any $\tau\in\Sigma$.
 Thus, $u\in U_0$ with $\sigma=Gu$.
 
 For the proof of equivalence of norms,
 let $C_a$ denote the continuity constant of $a$ with respect to 
 $\|\cdot\|_\Sigma$.
 The inf-sup condition \eqref{e:infsup} shows for any $v\in U_0$
 that
 $$
  \beta \|v\|_U
  \leq 
  \sup_{\tau\in\Sigma\setminus\{0\}} \frac{b(\tau,v)}{\|\tau\|_\Sigma}
  =
  \sup_{\tau\in\Sigma\setminus\{0\}} \frac{a(Gv,\tau)}{\|\tau\|_\Sigma}
  \leq 
  C_a^{1/2}\sup_{\tau\in\Sigma\setminus\{0\}} \frac{a(Gv,\tau)}{\|\tau\|_a}
  =
  C_a^{1/2}\|Gv\|_a.
 $$
 This implies the asserted equivalence of norms.
\end{proof}

The following compatibility condition is essential to 
the subsequent arguments.

\begin{condition}\label{cond:incl}
 The kernel
$$
 Z:= \{\tau\in\Sigma : b(\tau,v)=0 \text{ for all }v\in U\}
$$
and the discrete kernel
$$
Z_h:= \{\tau_h\in\Sigma_h : b(\tau_h,v_h)=0 \text{ for all }v_h\in U_h\}
$$
satisfy the inclusion $Z_h \subseteq Z$.
\end{condition}

It is known \cite[Lemma~2.3]{CarstensenGallistlSchedensack2016}
that Condition~\ref{cond:incl} is equivalent to the existence of a projection
$P_h:U\to U_h$ such that
\begin{equation}
\label{e:projP}
 b(\tau_h,v-P_h v) = 0
 \quad\text{for all }(\tau_h,v)\in\Sigma_h\times U.
\end{equation}
The argument is briefly repeated here for convenience of the reader.

\begin{lemma}
 Condition~\ref{cond:incl} is equivalent to the existence of
a linear projection $P_h:U\to U_h$ satisfying \eqref{e:projP}.
\end{lemma}
\begin{proof}
The closed range theorem \cite{Braess2007} in the finite-dimensional
setting states
$$
 b(\cdot,U_h) = Z_h^0 \subseteq \Sigma_h^*,
$$
namely that the functionals $b(\cdot,U_h)$
 in the dual $\Sigma_h^*$ of $\Sigma_h$ represented by 
elements of $U_h$ and the form $b$
are precisely the elements of the polar set 
$Z_h^0$ of the kernel $Z_h$,
i.e., the bounded linear functionals over $\Sigma_h$ that 
vanish on $Z_h$.
From the inclusion of kernels
$Z_h\subseteq Z$ in Condition~\ref{cond:incl}
we have that, given any $u\in U$, the functional
$b(\cdot,u)\in\Sigma_h^*$ vanishes on $Z_h$
and therefore belongs to $Z_h^0$,
whence $b(\cdot,u)\in b(\cdot,U_h)$.
Hence, there exists some $u_h\in U_h$ with
$b(\cdot,u)=b(\cdot,u_h)$.
This proves the existence of the projection $P_h u:=u_h$.
The element $P_h u$ can be chosen from the range of the discrete
Riesz map $T_h:\Sigma_h\to U_h$ representing the from $b$
via $b(\tau_h,\cdot)=(T_h\tau_h,\cdot)_U$ on $U_h$ for any
$\tau_h\in\Sigma_h$. With this choice, the projection $P_h$ is linear.
Conversely,
it is immediate that the existence of the projection $P_h$ implies
Condition~\ref{cond:incl}.
\end{proof}

\begin{example}
The example relevant to the applications in this paper is
the following.
Given $\tau\in\Sigma$, let $T\tau\in U$
denote the Riesz representation of $b(\tau,\cdot)$
in $U$ such that
any $v\in U$ satisfies
$b(\tau,v) = (T\tau,v)_U$.
Provided the inclusion $T\Sigma_h\subseteq U_h$
is satisfied,
the projection $P_h:U\to U_h$
can be chosen as the orthogonal projection 
$P_h$ to $U_h$ with
respect to the inner product of $U$.
\end{example}

Let $\mathfrak P_h:\Sigma\to\Sigma_h$ denote the orthogonal projection
to the finite-dimensional space $\Sigma_h$ with respect to the inner product $a$.

\begin{lemma}\label{l:commut}
If Condition~\ref{cond:incl} holds, then
 any $v\in U_0$ satisfies $G_h P_h v = \mathfrak P_h G v$.
\end{lemma}
\begin{proof}
 Given $v\in U_0$, the definition of $G_h$ and
 \eqref{e:projP}
 show for any $\tau_h\in\Sigma_h$ that
$$
   a(G_h P_h v,\tau_h)
   =
   -b(\tau_h,P_h v)
   =
   -b(\tau_h,v)
   =
   a(Gv,\tau_h).
$$
The last expression equals $a(\mathfrak P_h Gv,\tau_h)$
and the lemma ensues.
\end{proof}

\section{Approximation of the eigenvalue problem}\label{s:lower}

We adopt the setting of Section~\ref{s:commut}
with the forms $a$ and $b$, which are the ingredients of
problems in the well-known saddle-point structure.
It is assumed that $b$ satisfies the inf-sup condition
\eqref{e:infsup}.
It is additionally assumed that $U_0$ is dense in $U$.
Let furthermore
$$
 c,\ell:U\times U\to\mathbb R
$$
be two symmetric positive-semidefinite 
and bounded bilinear forms on $U$.
In order to exclude the totally trivial case $\ell=0$ 
(which would correspond to all eigenvalues equal to $+\infty$
in the system \eqref{e:evp} below) it is
assumed that there exists some $v_h\in U_h$ such that
$\ell(v_h,v_h)>0$ (which 
can be interpreted as a minimal resolution
condition on the discrete space).
The seminorms induced by $c$ and $\ell$ are denoted by
$\|\cdot\|_c = c(\cdot,\cdot)^{1/2}$ and
$\|\cdot\|_\ell = \ell(\cdot,\cdot)^{1/2}$.

The eigenvalue problem seeks eigenpairs
$(\lambda,u)\in\mathbb R\times U$ with nonzero $u$
such that
\begin{subequations}\label{e:evp}
\begin{align}\label{e:evp_a}
a(\sigma,\tau) + b(\tau,u) &= 0 &&\text{for all }\tau\in\Sigma\\
b(\sigma,v) -c(u,v)&= -\lambda \, \ell (u,v) 
       &&\text{for all }v \in U
.
\end{align}
\end{subequations}
Note that the variable $\sigma=Gu$ is determined by $u$
through \eqref{e:evp_a} and thus not treated as an independent
variable.
Recall Lemma~\ref{l:hilbert}, which states that $U_0$
is a Hilbert space.
Since, for any $v\in U_0$, $Gv$ and $v$ satisfy
$a(\sigma,Gv) = -b(\sigma,v)$
and the space $U_0$ is dense in $U$,
the eigenvalue problem is equivalent to seeking
eigenpairs $(\lambda,u)\in\mathbb R\times U_0$
satisfying, for all $v\in U_0$,
$$
\mathcal A(u,v) = \lambda\, \ell(u,v)
\quad\text{where}\quad \mathcal A(u,v):=a(Gu,Gv)+c(u,v).
$$
The left-hand side defines an inner product on $U_0$
and, hence, the eigenfunctions $u$ corresponding to
finite eigenvalues are $\mathcal A$-orthogonal
to the kernel of $\ell$ and will henceforth
be normalized as $\|u\|_\ell=1$.
We assume that the solution operator mapping
$f\in U$ to the solution $(\sigma,u)\in\Sigma\times U$
to the linear problem
\begin{align*}
a(\sigma,\tau) + b(\tau,u) &= 0 &&\text{for all }\tau\in\Sigma\\
b(\sigma,v) -c(u,v)&= -\ell (f,v) 
&&\text{for all }v \in U
\end{align*}
is a compact operator
(assuming the solution being measured in the 
norm $(\|u\|_U^2+\|\sigma\|_a^2)^{1/2}$).
Therefore, the finite part of the spectrum consists of
eigenvalues that have no finite accumulation point and can
be enumerated
$0<\lambda_1\leq\lambda_2\leq\dots$.
There exists an orthonormal set of corresponding eigenfunctions,
which will henceforth be referred to as ``the eigenfunctions''.
The Rayleigh quotient for the smallest eigenvalue
reads
$$
\lambda_1 = \min_{\substack{v\in U_0\setminus\{0\}\\ 
                  v\perp_{\mathcal A} \operatorname{ker}\ell}}
              \frac{\|Gv\|_a^2+\|v\|_c^2}{\|v\|_\ell^2}.
$$
Here, $\perp_{\mathcal A}$ denotes orthogonality with respect 
to $\mathcal A$
and $\operatorname{ker}\ell$ is the space of all
$v\in U$ with $\ell(v,v)=0$.
The higher eigenvalues satisfy analogous min-max
principles, 
the discrete version of which is displayed 
as \eqref{e:rayleigh_d_high} below.

For the choice of discrete spaces, we assume
that $b$ satisfies a discrete inf-sup condition
$$
0<\beta_h=
 \inf_{v_h\in U_h \setminus \{0\}}
 \sup_{\tau_h\in\Sigma_h\setminus \{0\}}
  \frac{b(\tau_h,v_h)}{\|\tau_h\|_\Sigma \|v_h\|_U} 
$$
with respect to the finite-dimensional spaces $\Sigma_h$ and $U_h$.
The discrete eigenvalue problem seeks discrete eigenpairs
$(\lambda_h,u_h)\in\mathbb R\times U_h$ with
$\|u_h\|_\ell=1$ such that
\begin{subequations}\label{e:devp}
\begin{align}
a(\sigma_h,\tau_h) + b(\tau_h,u_h) &= 0 
   &&\text{for all }\tau_h\in\Sigma_h \\
b(\sigma_h,v_h) -c(u_h,v_h)&= -\lambda_h \,\ell(u_h,v_h) 
  &&\text{for all }v_h\in U_h.
\end{align}
\end{subequations}
Analogous arguments as above show that
the discrete eigenvalue problem is equivalent to
$$
\mathcal A_h(u_h,v_h) = \lambda_h \,\ell(u_h,v_h)
\quad\text{where}\quad
\mathcal A_h(u_h,v_h):=a(G_h u_h,G_h v_h)+c(u_h,v_h).
$$
The bilinear form $\mathcal A_h$ on the left-hand side
defines an inner product on $U_h$.
The finite discrete eigenvalues are enumerated as
$0<\lambda_{1,h}\leq\lambda_{2,h}\leq\dots\leq\lambda_{N,h}$
for some positive integer $N$.
The first discrete eigenvalue minimizes the following
Rayleigh quotient
\begin{align*}
\lambda_{1,h} = \min_{\substack{v_h\in U_h\setminus\{0\}\\ 
                  v_h\perp_{\mathcal A_h} \operatorname{ker}\ell}}
              \frac{\|G_h v_h\|_a^2+\|v_h\|_c^2}{\|v_h\|_\ell^2}
\end{align*}
where $\perp_{\mathcal A_h}$ denotes orthogonality with respect 
to $\mathcal A_h$.
More generally,
the $J$th discrete eigenvalue satisfies the min-max principle
\cite{WeinsteinStenger1972}
\begin{equation}\label{e:rayleigh_d_high}
 \lambda_{J,h}
 =
\min_{\substack{V_J \subseteq  U_h
       \\ 
       \dim(V_J)=J, V_J\perp_{\mathcal A_h}\operatorname{ker}\ell}
      }
      \max_{v_h\in V_J\setminus\{0\}}
     \frac{\|G_h v_h \|_a^2 + \|v_h \|_c^2}{\|v_h\|_\ell^2} 
\end{equation}
where the minimum is taken over all $J$-dimensional subspaces of $U_h$
that are $\mathcal A_h$-orthogonal to the kernel of $\ell$.
Sufficient conditions on the spaces $\Sigma_h$ and $U_h$
such that the discrete eigenvalues $\lambda_{j,h}$ approximate
the true eigenvalues $\lambda_j$ are well known
\cite{Boffi2010}. The focus of this work is 
the computation
of guaranteed lower bounds to the eigenvalues $\lambda_j$
for an index $j\in\{1,\dots,J\}$.

The key condition required for the theory in this paper
is the following.

\begin{condition}\label{cond:main}
There exists a number $\delta_h=\delta_h(\Sigma_h,U_h)$ such
that any element
$u\in\operatorname{span}\{u_1,\dots,u_J\}$ in the linear hull
of the first $J$ eigenfunctions
satisfies
$$
 \|u-P_h u\|_\ell^2
 \leq \delta_h^2 ( \|Gu\|_a^2+\|u\|_c^2).
$$
\end{condition}

It is furthermore required that the projection $P_h$
is compatible with $c$ and $\ell$ in the following sense.
\begin{condition}\label{cond:proj}
Any element
$u\in\operatorname{span}\{u_1,\dots,u_J\}$ in the linear hull
of the first $J$ eigenfunctions satisfies
the Pythagorean identity
\begin{align}\label{e:proj_i}\tag{i}
 \|P_h u\|_\ell^2 + \|u-P_h u\|_\ell^2 = \|u\|_\ell^2,
\end{align}
the stability estimate
\begin{align}\label{e:proj_ii}\tag{ii}
\| P_h u\|_c^2 \leq \| u\|_c^2,
\end{align}
and the orthogonality
\begin{align}\label{e:proj_iii}\tag{iii}
  P_h u \perp_{\mathcal A_h}\operatorname{ker}\ell.
\end{align}
\end{condition}

\begin{remark}
 In all practical examples listed in this paper the 
 requirements from
 Condition~\ref{cond:proj} can be verified for 
 any $u\in U$.
\end{remark}

\begin{theorem}[abstract lower bound]\label{t:lowerbound}
 Let $\Sigma$, $U$ be Hilbert spaces with a symmetric and
 positive definite bilinear form $a$ on $\Sigma$ and a continuous
 bilinear form $b:\Sigma\times U\to\mathbb R$ such that
 the space $U_0\subseteq U$ of admissible right-hand
 sides for \eqref{e:gradient} is dense.
 Let the inf-sup condition \eqref{e:infsup},
 the inclusion of kernels from Condition~\ref{cond:incl}
 as well as Condition~\ref{cond:main} and Condition~\ref{cond:proj}
 be satisfied.
 Let $\Sigma_h\subseteq\Sigma$ and $U_h\subseteq U$ be
 an inf-sup stable pair of finite-dimensional subspaces
 and
 let $c$ and $\ell$ be continuous bilinear forms on $U$
 such that $\ell$ acts nontrivially on $U_h$
 and there exist at least $J$ finite 
 discrete eigenvalues $\lambda_{1,h},\dots,\lambda_{J,h}$
 to \eqref{e:devp}
 for a positive integer $J$.
 Let $\lambda_J$ denote the $J$th eigenvalue to
 \eqref{e:evp} with an $\ell$-normalized eigenfunction
 $u_J$.
 Then, the following lower bound for
  $\lambda_J$ holds
 $$
\frac{\lambda_{J,h}}{1+\delta_h^2\lambda_{J,h}} \leq  \lambda_J .
 $$
\end{theorem}
\begin{proof}
Let $u_1,\dots,u_J$ denote a basis of 
$\ell$-normalized eigenfunctions corresponding to the eigenvalues
$\lambda_1,\dots,\lambda_J$.
In a first step, we assume that
the projected eigenfunctions $(P_h u_j:j=1,\dots,J)$
are linear independent.
Since the  projected eigenfunctions
are linear independent and, by Condition~\ref{cond:proj}(iii)
they are $\mathcal A_h$-orthogonal to the kernel of $\ell$,
they form an admissible $J$-dimensional subspace 
$\tilde V_J$
and there are coefficients $\alpha_1,\dots,\alpha_J$,
normalized to $\sum_j\alpha_j^2 =1$,
such that
$$
 v_h := P_h u \quad\text{for } u:=\sum_{j=1}^J \alpha_j u_j
$$
maximizes the Rayleigh quotient over $\tilde V_J$.
The discrete Rayleigh quotient \eqref{e:rayleigh_d_high},
elementary properties of the minimum,
Lemma~\ref{l:commut},
the nonexpansivity of orthogonal projections,
and Condition~\ref{cond:proj}\eqref{e:proj_ii}
imply
$$
\|v_h \|_\ell^2 \lambda_{J,h}
\leq
     \|G_h v_h \|_a^2 +\| v_h \|_c^2
 \leq 
     \Big\|\sum_{j=1}^J \alpha_j G u_j\Big\|_a^2 
   + \Big\|\sum_{j=1}^J \alpha_j u_j \Big\|_c^2
     =\lambda_J 
$$
because different eigenfunctions are mutually
$\mathcal A$-orthogonal.
Condition~\ref{cond:proj},
Condition~\ref{cond:main},
and the normalization of the coefficients $\alpha_j$
show
$$
 \|v_h \|_\ell^2
 =
 \|u\|_\ell^2-\|u-P_h u\|_\ell^2
 \geq
 1 -\delta_h^2\lambda_J.
$$
The combination of the two foregoing displayed formulas
results in
$$
(1-\delta_h^2\lambda_J)\lambda_{J,h}\leq \lambda_J.
$$
Rearranging this formula yields the asserted lower eigenvalue
bound.

In the remaining case that
the projected eigenfunctions $(P_h u_j:j=1,\dots,J)$
are linear dependent,
we apply an idea from \cite{TanakaTakayasuLiuOishi2014}.
There exist coefficients
$\beta_1,\dots,\beta_J$ with $\sum_{j=1}^J\beta_j^2=1$
and a linear combination
$u:=\sum_{j=1}^J\beta_j u_j$ with $\|u\|_\ell=1$
and $P_hu=0$.
Condition~\ref{cond:main} implies
$$
1=\|u\|_\ell^2
=
\|u-P_hu\|_\ell^2
\leq 
 \delta_h^2 ( \|Gu\|_a^2+\|u\|_c^2)
\leq 
\delta_h^2 \lambda_J.
$$
The last inequality follows from the $\mathcal A$-orthogonality
of the eigenfunctions and the normalization of the 
coefficients $\beta_j$.
Since, in particular, $\delta_h>0$, we infer
$1/\delta_h^2\leq\lambda_J$.
Elementary estimates lead to
$$
\frac{\lambda_{J,h}}{1+\delta_h^2\lambda_{J,h}}
\leq 
\frac{\lambda_{J,h}}{\delta_h^2\lambda_{J,h}}
=
\frac{1}{\delta_h^2}
\leq \lambda_J,
$$
which implies the asserted lower eigenvalue bound.
\end{proof}

\section{Application to the Laplacian}\label{s:lapl}

This section is to fix notation and to present
the application of Theorem~\ref{t:lowerbound}
to the eigenvalues of the Laplacian.
Let $\Omega\subseteq\mathbb R^n$ be an open, bounded,
connected,
polytopal Lipschitz domain.
The eigenvalue problem for the Dirichlet-Laplacian
seeks eigenpairs $(\lambda,u)$ with
$$
 -\Delta u = \lambda u \quad\text{in }\Omega
 \quad u=0 \text{ on }\partial\Omega
$$
where the eigenfunction $u\in H^1_0(\Omega)\setminus\{0\}$
belongs to the first-order $L^2$-based Sobolev space
with vanishing trace on the boundary,
and the Laplacian $\Delta$ is understood in the sense of
weak derivatives.
The mixed formulation 
is based on the choice
$$
 \Sigma=H(\ddiv,\Omega) \quad\text{and}\quad
 U=L^2(\Omega)
$$
where $L^2(\Omega)$ is the space of square-integrable
measurable functions and $H(\ddiv,\Omega)$ is the space
of vector fields over $\Omega$ whose components as well as
distributional divergence belong to $L^2(\Omega)$.
The inner product in (any power of) $L^2(\Omega)$
is denoted by $(\cdot,\cdot)_{L^2(\Omega)}$.
The form $a$ is defined as
the $L^2$ inner product of vector fields,
$a(\cdot,\cdot):=(\cdot,\cdot)_{L^2(\Omega)}$,
and
$b$ is defined by
$$
   b(\tau,v) := (\ddiv\tau,v)_{L^2(\Omega)}
  \quad\text{for any }(\tau,v)\in\Sigma\times U.
$$
With the choice $c=0$ and 
$\ell(\cdot,\cdot)=(\cdot,\cdot)_{L^2(\Omega)}$,
it is well known \cite{Boffi2010,BoffiBrezziFortin2013} that
the eigenvalues of the Laplacian correspond to those of
system \eqref{e:evp}
and that the form $b$ satisfies the inf-sup condition
\eqref{e:infsup}.

Let $\Sigma_h$ and $U_h$ be an inf-sup stable pair of
finite-dimensional spaces
with the property $\ddiv\Sigma_h\subseteq U_h$,
which guarantees Condition~\ref{cond:incl}.
It is assumed that the discrete spaces are related to
a partition $\mathcal T$ of $\bar\Omega$ in convex polytopes
(for example a simplicial triangulation)
and that the space $P_0(\mathcal T)$ of piecewise constant
functions over $\mathcal T$ is contained in $U_h$.
The most prominent example of such a pair is the choice
of $\Sigma_h$ as the lowest-order Raviart--Thomas space
with respect to $\mathcal T$ and $U_h=P_0(\mathcal T)$,
but many other choices are possible \cite{BoffiBrezziFortin2013}.
Since the piecewise constants are contained in $U_h$,
we have
$$
 \| u-P_h u\|_{U}
 \leq
 \| u-\Pi_{0,h} u\|_{L^2(\Omega)}
$$
for the $L^2$ projection $\Pi_{0,h}$ onto the piecewise constants.
Each element $T\in\mathcal T$ of the partition
is convex, whence the constant of the Poincar\'e inequality
is explicitly known \cite{PayneWeinberger1960}
and equals $h_T/\pi$ with the diameter
$h_T:=\operatorname{diam}(T)$ of $T$.
Therefore
$$
 \| u-P_h u\|_{U}
 \leq 
 \frac{h}{\pi} \| \nabla u\|_{L^2(\Omega)}
=
 \frac{h}{\pi} \| Gu\|_{L^2(\Omega)}
$$
for the maximal element diameter $h:=\max_{T\in\mathcal T} h_T$,
where it has been used that $u\in U_0$ possesses a weak
gradient in $[L^2(\Omega)]^n$.
This verifies Condition~\ref{cond:main} with
$\delta_h = h/\pi$.
Condition~\ref{cond:proj} is trivially satisfied
because $\ell$ is the $L^2$ inner product and $c=0$.
In conclusion, Theorem~\ref{t:lowerbound} applies
and the resulting lower bound for the Laplacian reads as follows.

\begin{corollary}[guaranteed lower eigenvalue bound for the Laplacian]
\label{c:lap}
Assume the above setting for the mixed formulation of
the Dirichlet-Laplacian.
Let $\Sigma_h\subseteq\Sigma$, $U_h\subseteq U$ be an inf-sup
stable pair of finite-dimensional subspaces related to a
partition $\mathcal T$ in convex polytopes
with $\mathcal P_0(\mathcal T)\subseteq U_h$
and $\ddiv\Sigma_h\subseteq U_h$.
Then, the $J$th eigenvalue $\lambda_J$ 
of \eqref{e:evp} and the $J$th
discrete eigenvalue $\lambda_{J,h}$ of \eqref{e:devp} satisfy
 $$
\frac{\lambda_{J,h}}{1+(h^2/\pi^{2})\lambda_{J,h}}
\leq  \lambda_J .
$$
\end{corollary}

\begin{remark}\label{r:bessel}
It is known that for $n=2$ and triangular partitions
the constant of the Poincar\'e inequality can be
slightly improved \cite{LaugensenSiudeja2010}.
In this case, $\pi^2$ in Corollary~\ref{c:lap}
can be replaced by $j_{1,1}^2$ where $j_{1,1}$ is the 
first root of the Bessel function of the first kind.
\end{remark}

\begin{example}\label{ex:lapl}
Consider the first eigenvalue of the Dirichlet-Laplacian
on the L-shaped domain $(-1,1)^2\setminus[0,1]^2$.
Let $\mathcal T$ be a triangulation of $\Omega$
and let
$\Sigma_h$ be the lowest-order Raviart--Thomas finite element
space \cite{BoffiBrezziFortin2013},
which is the subspace of $H(\ddiv,\Omega)$ of vector
fields that, when restricted to any $T\in\mathcal T$,
are linear combinations of constants and the identity
$x\mapsto x$.
The corresponding space $U_h=P_0(\mathcal T)$ is the space 
of piecewise constants.
The initial triangulation is displayed in
Figure~\ref{f:L}.
Table~\ref{tab:laplace} displays
the discrete eigenvalue, the guaranteed lower bound
from Corollary~\ref{c:lap},
and an upper bound computed with a first-order conforming
FEM
on a sequence of uniformly refined meshes.
The computed bound is that from Corollary~\ref{c:lap}
and disregards the slight improvement mentioned in 
Remark~\ref{r:bessel}.

\begin{figure}
\begin{tikzpicture}
  \draw (-1,-1)--(1,-1)--(1,0)--(0,0)--(0,1)--(-1,1)--cycle;
  \draw (0,1)--(-1,0)--(1,0)--(0,-1)--(0,1);
  \draw (-1,-1)--(0,0);
  \draw (-.5,.5)--(-.5,1)--(-1,.5)--(-.5,.5)--(0,.5)--(-.5,0)--cycle;
%  \draw (.5,.5)--(.5,1)--(0,.5)--(.5,.5)--(.5,0)--(1,.5)--cycle;
 \draw (-.5,-.5)--(-.5,-1)--(0,-.5)--(-.5,-.5)--(-.5,0)--(-1,-.5)--cycle;
 \draw (.5,-.5)--(.5,-1)--(1,-.5)--(.5,-.5)--(0,-.5)--(.5,0)--cycle;
 \end{tikzpicture}
 \caption{Initial triangulation of the L-shaped domain.
          \label{f:L}}
\end{figure}
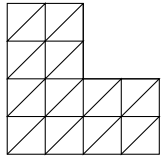

\begin{table}
 \begin{tabular}{cccc}
  $h\times\sqrt{2}$ & $\lambda_{1,h}$ & lower bound & upper bound\\
  \hline   
   $2^0$    &  8.60144 &  5.99088 &  13.1991  \\
   $2^{-1}$ &  9.25186 &  8.28147 &  10.5739  \\
   $2^{-2}$ &  9.49208 &  9.21512 &  9.91654  \\
   $2^{-3}$ &  9.58268 &  9.51054 &  9.72837  \\
   $2^{-4}$ &  9.61746 &  9.59919 &  9.66981  \\
 \end{tabular}
\caption{Results for the Laplace eigenvalues on the L-shaped domain.
\label{tab:laplace}}
\end{table}
\end{example}

\section{Scalar elliptic operator}\label{s:elliptic}
As a generalization of the eigenvalue problem from the
previous section we consider the eigenvalue problem
\begin{align}\label{e:scalarellipt}
 -\ddiv (A\nabla u) + \gamma u
   = \lambda u\quad\text{in }\Omega
 \quad u=0 \text{ on }\partial\Omega .
\end{align}
Here $A$ is a symmetric matrix field over $\Omega$
with $L^\infty(\Omega)$ coefficients satisfying the bounds
$$
  a_0 |\xi|^2 \leq \xi^T A \xi \leq a_1 |\xi|^2
 \quad\text{for any }\xi\in\mathbb R^n
 \quad\text{a.e. in }\Omega
$$
with real numbers $0<a_0\leq a_1 <\infty$;
and $\gamma\in L^\infty(\Omega)$ is a 
nonnegative function
with $0\leq\gamma_0\leq\gamma\leq\gamma_1$ almost everywhere.
As in the Laplacian case, the spaces for the mixed
formulation are 
$$
 \Sigma=H(\ddiv,\Omega) \quad\text{and}\quad
 U=L^2(\Omega) .
$$
For simplicity it is assumed that $A$ and $\gamma$ are
piecewise constant with respect to a given triangulation
$\mathcal T$, which will also be used for the discretization.

The  mixed formulation is based on the
substitution $\sigma=A\nabla u$.
The form $a$ is defined as
$a(\cdot,\cdot):=(\cdot,A^{-1}\cdot)_{L^2(\Omega)}$,
and
$b$ is defined by
$$
   b(\tau,v) := (\ddiv\tau,v)_{L^2(\Omega)}
  \quad\text{for any }(\tau,v)\in\Sigma\times U.
$$
With the choice 
$c(\cdot,\cdot)=(\cdot,\gamma\cdot)_{L^2(\Omega)}$
and
$\ell(\cdot,\cdot)=(\cdot,\cdot)_{L^2(\Omega)}$,
it is not difficult to verify that 
system~\eqref{e:evp} is then
inf-sup stable and equivalent to
the original problem \eqref{e:scalarellipt}.

Let $\Sigma_h$ and $U_h$ be an inf-sup stable pair of
finite-dimensional spaces
with the property $\ddiv\Sigma_h\subseteq U_h$.
It is again assumed that the discrete spaces are related to
a partition $\mathcal T$ of $\bar\Omega$ in convex polytopes
and that $U_h$ contains the piecewise constant functions
$P_0(\mathcal T)$.
In addition, it is assumed that $U_h$ does not include a constraint
on inter-element continuity.
More precisely, it is assumed that $U_h$ is of the structure
\begin{equation}
\label{e:prodstruct}
  U_h = \prod_{T\in\mathcal T} V_T
\end{equation}
where $V_T$ is a subspace of $L^2(T)$ and the embedding 
$L^2(T)\subseteq L^2(\Omega)$ is understood through extensions by zero.
This assumption ensures that the $L^2$ projection to $U_h$
localizes to the elements of $\mathcal T$.
This property is used
for the verification of the stability property
from Condition~\ref{cond:proj}\eqref{e:proj_ii}:
Since
the $L^2$ projection onto $U_h$ equals the local $L^2$ projection,
we have
$$
\| P_h u\|_c^2
\leq
\sum_{T\in\mathcal T} \gamma|_T
 \left\|\Pi_{h,T}  u \right\|_{L^2(T)}^2
\leq 
\sum_{T\in\mathcal T} \gamma|_T \|u\|_{L^2(T)}^2
=
\| u\|_c^2 .
$$
For verifying Condition~\ref{cond:main}
and determining the constant,
we use the local Poincar\'e inequality and infer
$$
\| u-P_h u\|_\ell^2
=
\| u-\Pi_{0,h}u\|_{L^2(\Omega)}^2
\leq 
\frac{h^2}{\pi^2} \| \nabla u\|_{L^2(\Omega)}^2
\leq 
\frac{h^2}{a_0\pi^2} \| A^{1/2}\nabla u\|_{L^2(\Omega)}^2
$$
so that Condition~\ref{cond:main} is satisfied with
$\delta_h=h/(a_0^{1/2}\pi)$.
Theorem~\ref{t:lowerbound} therefore implies the following
result.

\begin{corollary}[guaranteed lower eigenvalue bound for elliptic operators]
\label{c:ellipt}
Assume the above setting for the mixed formulation of
\eqref{e:scalarellipt}.
Let $\Sigma_h\subseteq\Sigma$, $U_h\subseteq U$ be an inf-sup
stable pair of finite-dimensional subspaces related to a
partition $\mathcal T$ in convex polytopes
with $\mathcal P_0(\mathcal T)\subseteq U_h$
and $\ddiv\Sigma_h\subseteq U_h$
where $U_h$ is a space without interelement continuity
requirements in the sense of the structure from
\eqref{e:prodstruct}.
Then, the $J$th eigenvalue $\lambda_J$ 
of \eqref{e:evp} and the $J$th
discrete eigenvalue $\lambda_{J,h}$ of \eqref{e:devp} satisfy
 $$
\frac{\lambda_{J,h}}{1+\lambda_{J,h} h^2/(a_0\pi^2)}
\leq  \lambda_J .
$$
\end{corollary}

In the case that the lower bound $\gamma_0$ to the 
low-order coefficient $\gamma$ in the elliptic eigenvalue
problem is positive, one can take advantage of a
spectral shift and obtain a sharper lower bound.

\begin{corollary}[guaranteed lower eigenvalue bound for elliptic operators with shift]
\label{c:ellipt_shift}
Under the assumptions of Corollary~\ref{c:ellipt},
the $J$th eigenvalue $\lambda_J$ 
of \eqref{e:evp} and the $J$th
discrete eigenvalue $\lambda_{J,h}$ of \eqref{e:devp} satisfy
 $$
\frac{\lambda_{J,h}}{1+(\lambda_{J,h}-\gamma_0) h^2/(a_0\pi^2)}
+
\gamma_0 \frac{(\lambda_{J,h}-\gamma_0) h^2/(a_0\pi^2)}{1+(\lambda_{J,h}-\gamma_0) h^2/(a_0\pi^2)}
\leq  \lambda_J .
$$
\end{corollary}
\begin{proof}
The eigenvalues $\hat\lambda$ of the shifted problem
 \begin{align*}
 -\ddiv (A\nabla u) + (\gamma -\gamma_0) u
   = \hat \lambda u\quad\text{in }\Omega
 \quad u=0 \text{ on }\partial\Omega 
\end{align*}
are related to the ones of \eqref{e:scalarellipt}
by $\hat\lambda_j+\gamma_0=\lambda_j$ and an analogous shift
property applies to the discrete problem
so that $\hat\lambda_{j,h}+\gamma_0=\lambda_{j,h}$.
Since $\gamma-\gamma_0$ is nonnegative, 
Corollary~\ref{c:ellipt} applies and proves
 $$
\frac{\hat \lambda_{J,h}}{1+\hat\lambda_{J,h}h^2/(a_0\pi^2)}
\leq  \hat\lambda_J .
$$
Equivalently,
 $$
\frac{\lambda_{J,h}-\gamma_0}{1+(\lambda_{J,h}-\gamma_0)h^2/(a_0\pi^2)}
+\gamma_0
\leq \lambda_J .
$$
This implies the asserted lower bound.
\end{proof}

\begin{example}\label{ex:scalarellipt}
On the square domain $\Omega=(-1,1)^2$ choose the coefficients
$$
 A(x) = \left(2+ \frac{x_1 x_2}{|x_1|\,|x_2|}\right) I_{2\times 2}
 \quad\text{and}\quad
 \gamma(x) = 4 + 1_{\{|x_2|>1/2\}}
$$
where $I_{2\times 2}$ is the two-dimensional unit matrix.
The lower bounds on the coefficients read $a_0=1$ and $\gamma_0=4$.
The coefficients and the initial triangulation
are displayed in Figure~\ref{f:scalarcoeff}.
Table~\ref{tab:elliptic} compares the discrete eigenvalues,
the guaranteed lower bound
from Corollary~\ref{c:ellipt_shift},
and upper bounds from a conforming
standard FEM on a sequence of uniformly refined meshes.

\begin{figure}
 \begin{tikzpicture}
  \draw (-1,-1)--(1,-1)--(1,1)--(-1,1)--cycle;
  \draw[dashed] (0,-1)--(0,1) (-1,0)--(1,0);
  \node at (-.5,-.5) {3};
  \node at (-.5,.5) {1};
  \node at (.5,.5) {3};
  \node at (.5,-.5) {1};
 \end{tikzpicture}
 \quad
\begin{tikzpicture}
  \draw (-1,-1)--(1,-1)--(1,1)--(-1,1)--cycle;
  \draw[dashed] (-1,-.5)--(1,-.5) (-1,.5)--(1,.5);
  \node at (0,.75) {5};
  \node at (0,0) {4};
  \node at (0,-.75) {5};
 \end{tikzpicture}
 \quad
\begin{tikzpicture}
  \draw (-1,-1)--(1,-1)--(1,1)--(-1,1)--(-1,1)--cycle;
  \draw (0,1)--(-1,0)--(1,0)--(0,-1)--(0,1);
  \draw (-1,-1)--(1,1);
  \draw (-.5,.5)--(-.5,1)--(-1,.5)--(-.5,.5)--(0,.5)--(-.5,0)--cycle;
 \draw (.5,.5)--(.5,1)--(0,.5)--(.5,.5)--(.5,0)--(1,.5)--cycle;
 \draw (-.5,-.5)--(-.5,-1)--(0,-.5)--(-.5,-.5)--(-.5,0)--(-1,-.5)--cycle;
 \draw (.5,-.5)--(.5,-1)--(1,-.5)--(.5,-.5)--(0,-.5)--(.5,0)--cycle;
 \end{tikzpicture}
 \caption{Coefficients $A$ (left)
        and $\gamma$ (middle)
        and the initial triangulation (right)
        in Example~\ref{ex:scalarellipt}.
          \label{f:scalarcoeff}}
\end{figure}
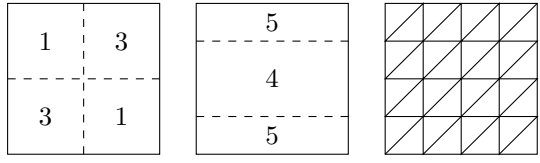

\begin{table}
 \begin{tabular}{cccc}
  $h\times\sqrt{2}$ & $\lambda_{1,h}$ & lower bound & upper bound\\
  \hline   
   $2^0$    &  13.4656 &  10.3977 &  15.4049 \\
   $2^{-1}$ &  13.4010 &  12.4008 &  13.9124 \\
   $2^{-2}$ &  13.3898 &  13.1187 &  13.5205 \\
   $2^{-3}$ &  13.3877 &  13.3185 &  13.4207 \\
   $2^{-4}$ &  13.3873 &  13.3699 &  13.3956 \\
 \end{tabular}
\caption{Results for eigenvalue problem \eqref{e:scalarellipt} of the 
         scalar elliptic operator on the square domain.
\label{tab:elliptic}}
\end{table}

\end{example}

\section{Application to linear elasticity}\label{s:elast}
Let $\Omega$ be a domain as in the previous sections
with a disjoint partition $\partial\Omega=\Gamma_D\cup\Gamma_N$
of the boundary where $\Gamma_D$ is assumed, for simplicity
of the presentation, to have positive surface measure.
For the sake of a simple exposition, it is furthermore assumed
that the parts $\Gamma_D$ and $\Gamma_N$ are resolved by 
the boundary faces of some underlying polytopal partition
$\mathcal T$ of $\bar\Omega$.
The linear elasticity eigenvalue problem seeks eigenvalues
$\lambda$ and vector-valued eigenfunctions $u\neq0$ such 
that
$$
 -\ddiv \mathbb C\varepsilon (u) = \lambda u
\quad\text{in }\Omega
\qquad\text{and}\qquad
u=0\quad\text{on }\Gamma_D
\qquad
\mathbb C\varepsilon (u)\mathbf n=0\quad\text{on }\Gamma_N.
$$
Here, 
$\mathbf n$ is the outward pointing unit vector to $\Gamma_N$,
$\varepsilon (u) = \frac12 (D u + (Du)^T)$ is the 
symmetric part of the derivative matrix,
and the elasticity tensor $\mathbb C$ reads
$$
  \mathbb C (A) 
   = 2\mu A 
   + \kappa \operatorname{tr}(A) I_{n\times n}
\quad\text{for any symmetric matrix }A
$$
for given material parameters $\mu,\kappa>0$
and the $n$-dimensional unit matrix $I_{n\times n}$.
The action of the divergence to a $n\times n$ matrix field is
understood row-wise.
The mixed formulation is based on the space
$\Sigma:=H_{\Gamma_N}(\ddiv,\Omega;\mathbb S)$ of symmetric matrix fields
$\sigma$
whose rows belong to $H(\ddiv,\Omega)$ and that
satisfy the homogeneous Neumann boundary condition
$\sigma\mathbf n =0$ on $\Gamma_N$;
and 
$U:=[L^2(\Omega)]^n$.
The bilinear forms $a$, $b$ are defined as
$a(\cdot,\cdot):=(\cdot,\mathbb C^{-1}\cdot)_{L^2(\Omega)}$,
and
$$
   b(\tau,v) := (\ddiv\tau,v)
  \quad\text{for any }(\tau,v)\in\Sigma\times U
$$
while $c=0$ and 
$\ell(\cdot,\cdot)=(\cdot,\cdot)_{L^2(\Omega)}$.
It is known that this is an inf-sup stable formulation of the 
linear elastic eigenvalue problem
\cite{ArnoldWinther2002,ArnoldFalkWinther2006,BoffiBrezziFortin2013}.

Let $\Sigma_h$ and $U_h$ be an inf-sup stable pair of
finite-dimensional spaces
with $\ddiv\Sigma_h\subseteq U_h$.
Instances of such spaces with pointwise symmetry for the 
stress field can be based on piecewise polynomials
\cite{ArnoldWinther2002,ArnoldFalkWinther2006}
or on piecewise rational trial functions
\cite{GuzmanNeilan2014}.
It is assumed that the discrete spaces are related to
a partition $\mathcal T$ of $\bar\Omega$ in convex polytopes
and that the space 
$$
 \Big\{ v\in [L^2(\Omega)]^n:
   \text{for any }T\in\mathcal T,
    D(v|_T)
   \text{ is constant and skew-symmetric}
 \Big\}
$$
of piecewise infinitesimal rigid body motions 
with respect to $\mathcal T$ is contained in $U_h$.
Since the infinitesimal rigid-body motions on an element
$\mathcal T$ include all
constants, the Poincar\'e inequality yields on any $T\in\mathcal T$
for the $L^2$ projection $\Pi_{\mathrm{RM},h}$ onto the 
piecewise infinitesimal rigid-body motions,
$$
\|u-\Pi_{\mathrm{RM},h} u\|_{L^2(T)}
\leq 
\frac{h_T}{\pi} \|D (u-\Pi_{\mathrm{RM},h}u)\|_{L^2(T)}.
$$
Korn's inequality on $T$ with constant $C_K(T)$ then yields
$$
\|u-\Pi_{\mathrm{RM},h} u\|_{L^2(T)}
\leq 
\frac{h_T}{\pi} \|D (u-\Pi_{\mathrm{RM},h}u)\|_{L^2(T)}
\leq 
\frac{C_K(T)h_T}{\pi} \|\varepsilon (u)\|_{L^2(T)}.
$$
Thus, with $1/(2\mu)$ as the smallest eigenvalue of the 
elasticity tensor $\mathbb C$ and
$$
\delta_h:=\frac{\max_{T\in\mathcal T} C_K(T) h_T}{\sqrt{2\mu}\pi}
$$
it follows that
$$
\|u-P_h u\|_{U}^ 2
\leq 
\frac{\max_{T\in\mathcal T} C_K(T)^2 h_T^2}{\pi^2}
 \|\varepsilon (u)\|_{L^2(\Omega)}^2
\leq 
\delta_h^2
 \|\mathbb C^{1/2}\varepsilon (u)\|_{L^2(\Omega)}^2
$$
which verifies Condition~\ref{cond:main}.

\begin{corollary}[guaranteed lower eigenvalue bound for elasticity]
\label{c:elast}
Assume the above setting for the mixed formulation of
the elasticity system.
Let $\Sigma_h\subseteq\Sigma$, $U_h\subseteq U$ be an inf-sup
stable pair of finite-dimensional subspaces related to a
partition $\mathcal T$ in convex polytopes
with $\ddiv\Sigma_h\subseteq U_h$
where $U_h$ contains the piecewise infinitesimal
rigid-body motions.
Then, the $J$th eigenvalue $\lambda_J$ 
of \eqref{e:evp} and the $J$th
discrete eigenvalue $\lambda_{J,h}$ of \eqref{e:devp} satisfy
 $$
\frac{\lambda_{J,h}}
 {1+(\max_{T\in\mathcal T} C_K(T)^2 h_T^2)/(2\mu\pi^2))\lambda_{J,h}}
\leq  \lambda_J .
$$
\end{corollary}

For practical and guaranteed bounds, upper bounds on the local
Korn constants $C_K(T)$ are needed.
In two dimensions, upper bounds for $C_K(T)$ 
on convex polygons can be explicitly computed
from the bound
on the continuity constant of a right-inverse of the 
divergence operator
available in \cite[Section~5.1.2]{CostabelDauge2015}.
It is known that the latter has a close relation to the 
Korn constant, and the precise argument is given as follows.
 Let $\omega\subseteq\mathbb R^n$ be an open, bounded,
 connected Lipschitz domain.
 It is well known 
 \cite{AcostaDuran2017}
 that there exists a constant
 $C_{\ddiv}<\infty$ such that any $p\in L^2(\Omega)$ with 
 $\int_\omega p\,dx =0$ can be represented as
 $p=\ddiv v$ with a vector field $v\in [H^1_0(\Omega)]^n$
 with $\|Dv\|_{L^2(\Omega)}\leq C_{\ddiv} \|p\|_{L^2(\Omega)}$.
If $\omega\subseteq\mathbb R^2$ is in addition a convex planar polygon
with corners $z_1,\dots,z_m$,
a fixed (arbitrary) interior point $x_0\in\omega$,
and the geometric parameter
\begin{align}
 \label{e:param_d}
 d:=\frac{\operatorname{dist}(x_0,\partial\omega)}
        {\max_{j=1,\dots,m} |x_0 - z_j|}, 
\end{align}
then the following bound
provided by \cite{CostabelDauge2015}
is valid
\begin{equation}\label{e:Cdivbound}
C_{\ddiv}
\leq 
\sqrt{
\frac{2}{d^2} (1+\sqrt{1-d^2})}  .
\end{equation}
In what follows, we will always choose
$x_0$ as the centre of the largest incribed ball of the
polyhedron.
Then, for example, any right-isosceles triangle $\hat T$ satisfies
$d(\hat T)= 1/\sqrt{4+2\sqrt{2}}$
and, accordingly,
$$
C_{\ddiv}(\hat T) \leq 5.1259 .
$$

The following lemma shows how the bound
on $C_{\ddiv}$ can be used for
bounding the Korn constant in two space dimensions.
The usual rotation of two-dimensional vector fields 
reads $\rot v = \partial_2 v_1 - \partial_1 v_2$.
\begin{lemma}[explicit bound on local Korn inequality in 2D]
 \label{l:korn}
 Let $\omega\subseteq\mathbb R^2$ be a bounded, open, convex
 polygon with the geometric parameter $d$ from \eqref{e:param_d}
 and let
 $v\in [H^1(\omega)]^2$ be a vector field with
 $\int_\omega \rot v \,dx = 0$.
 Then 
 $$
\|Dv\|_{L^2(\omega)}
 \leq 
\sqrt{1+\frac{4}{d^2} (1+\sqrt{1-d^2})}
\|\varepsilon (v)\|_{L^2(\omega)}.
$$
\end{lemma}
\begin{proof}
 The tensor field $\tau = Dv$ is irrotational,
 which is equivalent to the fact that the field
 $\tau^\perp:= (-\tau_{12},\tau_{11}; -\tau_{22},\tau_{21})$
 is divergence-free.
 Further, the property $\int_\omega \rot v \,dx= 0$ is
 equivalent to $\int_\omega \operatorname{tr}\tau^\perp\,dx =0$.
 By a classical argument \cite[Proposition~9.1.1]{BoffiBrezziFortin2013}
 it can be shown that 
 $$
  \|\operatorname{tr}\tau^\perp\|_{L^2(\Omega)}
  \leq 
  2C_{\ddiv} 
   \|\tau^\perp - \frac12 \operatorname{tr}\tau^\perp I \|_{L^2(\Omega)}
 .
 $$
 This implies, with
 $
  I_\perp:=
  \left(\begin{smallmatrix}
     0&-1\\1&0
    \end{smallmatrix}\right)
 $
that
$$
  \|\tau_{21}-\tau_{12}\|_{L^2(\Omega)}
  \leq 
  2C_{\ddiv} 
   \|\tau - \frac12 (\tau_{21}-\tau_{12}) 
    I_\perp
 \|_{L^2(\Omega)}
  = 2C_{\ddiv} \|\frac12 (\tau+\tau^T)\|_{L^2(\Omega)}
 $$
 so that the skew-symmetric part of $\tau$ is controlled by
 the symmetric part of $\tau$.
 From the orthogonality of symmetric and skew-symmetric matrices
 we then infer with $\varepsilon(v)=\frac12 (\tau+\tau^T)$ that
 $$
  \|\tau\|_{L^2(\Omega)}^2
  =
  \|\varepsilon (v)\|_{L^2(\Omega)}^2
  +
  \|\frac12 (\tau_{21}-\tau_{12}) I_\perp
  \|_{L^2(\Omega)}^2
 \leq 
  (1+2C_{\ddiv}^2) \|\varepsilon (v)\|_{L^2(\Omega)}^2.
 $$
 The asserted estimate then follows with the bound
\eqref{e:Cdivbound} on $C_{\ddiv}$.
\end{proof}

Lemma~\ref{l:korn} shows that the Korn constant
on a convex polygon $\omega$ can be bounded 
as
$$
 C_K\leq \sqrt{1+\frac{4}{d^2} (1+\sqrt{1-d^2})} .
$$
For example,
given a right-isosceles triangle $\hat T$, the value
of the Korn constant satisfies the bound
$$
C_K(\hat T) \leq 7.318 .
$$

\begin{remark}
 For domains with sufficiently regular boundary,
 there exists a sharper alternative to the the bound of Lemma~\ref{l:korn},
 see \cite{HorganPayne1983} and the references therein.
\end{remark}

\begin{example}
The Dirichlet boundary $\Gamma_D$ of Cook's membrane 
$\Omega\subseteq \mathbb R^2$
is given by the straight line from 
$(0,0)$ to $(0,44)$.
The domain $\Omega$ is given as the interior of the convex combination 
of $\Gamma_D$ with the points $(48,44)$ and $(48,60)$,
and, accordingly, $\Gamma_N=\partial\Omega\setminus\Gamma_D$ is
the Neumann boundary.
The domain with its initial triangulation is displayed in
Figure~\ref{f:cook}.
In the numerical example, the Lam\'e parameters are
chosen as $\mu=1$ and $\kappa=100$.
The spaces $\Sigma_h$ and $U_h$ are taken according to the
Arnold--Winther finite element method
\cite{ArnoldWinther2002} based on a regular triangulation
$\mathcal T$, that is,
$\Sigma_h$ is the subspace of $\Sigma$ consisting of symmetric
matrix fields whose components,
when restricted to any triangle of $T$, are at most cubic polynomials
on $T$, and whose divergence is piecewise affine,
while $U_h$ is the space of piecewise affine vector fields.
Table~\ref{tab:elast} displays
the discrete eigenvalue, the guaranteed lower bound
from Corollary~\ref{c:lap},
and an upper bound computed with a first-order conforming
FEM
on a sequence of uniformly refined meshes.
It should be remarked that the optimal order of convergence
of the Arnold--Winther method is better than linear,
so that in general the guaranteed lower bound including the 
global mesh size $h$ is expected to be sub-optimal
if sufficient smoothness of the eigenfunctions is available.
This effect is not visible in this experiment because of the 
Dirichlet--Neumann corners in the configuration of the boundary,
which lead to reduced regularity.
There exist lower-order methods respecting the symmetry of stresses
(see, e.g., \cite{ArnoldWinther2002,GuzmanNeilan2014}),
but their implementation is not necessarily easier compared with
the usual Arnold--Winther finite element.

 \begin{figure}
 \begin{tikzpicture}[scale=.05]
    \coordinate (a1) at (0.00000,0.00000);
    \coordinate (a2) at (22.36692,   20.50301);
    \coordinate (a3) at (32.62352,   29.90490);
    \coordinate (a4) at (43.01651,   39.43180);
    \coordinate (a5) at (48.00000,   44.00000);
    \coordinate (a6) at (48.00000,   52.00000);
    \coordinate (a7) at (48.00000,   60.00000);
    \coordinate (a8) at (34.18905,   55.39635);
    \coordinate (a9) at (20.96553,   50.98851);
    \coordinate (a10) at (7.91553 ,  46.63851);
    \coordinate (a11) at (0.00000 ,  44.00000);
    \coordinate (a12) at (0.00000 ,  22.00000);
    \coordinate (a13) at (13.34761,   30.34226);
    \coordinate (a14) at (25.34761,   37.84226);
    \coordinate (a15) at (37.50718,   45.44199);  
    \draw (a1)--(a13)--(a12)--cycle
   (a13)--(a1)--(a2)--cycle
    (a2)--(a14)--(a13)--cycle
   (a14)--(a2)--(a3)--cycle
    (a3)--(a15)--(a14)--cycle
   (a15)--(a3)--(a4)--cycle
    (a5)--(a15)--(a4)--cycle
   (a15)--(a5)--(a6)--cycle
    (a7)--(a15)--(a6)--cycle
   (a15)--(a7)--(a8)--cycle
    (a9)--(a15)--(a8)--cycle
   (a15) --(a9)--(a14)--cycle
   (a10)--(a14)--(a9)--cycle
   (a14)--(a10)--(a13)--cycle
   (a11)--(a13)--(a10)--cycle
   (a13)--(a11)--(a12)--cycle;
\draw[very thick] (0,0)--(0,44);
\draw[very thick,dashed] (0,44)--(48,60)--(48,44)--(0,0);
\node at (0,22)[left]{$\Gamma_D$};
\node at (24,22)[below right]{$\Gamma_N$};
 \end{tikzpicture}
\caption{Cook's membrane with initial trianulation.
        \label{f:cook}}
\end{figure}
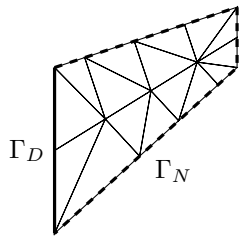
 \begin{table}
 \begin{tabular}{cccc}
  $h$ & $\lambda_{1,h}$ & lower bound & upper bound\\
  \hline   
   33.14 &   5.62812e-4 &  4.40411e-5  & 1.00153e-3  \\
   16.57 &   5.68410e-4 &  1.43028e-4  & 9.07262e-4  \\
   8.287 &   5.71724e-4 &  3.27099e-4  & 7.68356e-4  \\
   4.143 &   5.73587e-4 &  4.82990e-4  & 6.58116e-4  \\
   2.071 &   5.74507e-4 &  5.48734e-4  & 6.05333e-4  \\
 \end{tabular}
\caption{Results for the elasticity eigenvalues on 
          Cook's membrane.
\label{tab:elast}}
\end{table}
 \end{example}

\begin{remark}
A similar reasoning yields lower eigenvalue bounds for
the Stokes system, which corresponds to the formal
limit $\kappa\to\infty$.
It is known that the mixed formulation of the Lam\'e
system is robust (locking-free) with respect to this limit.
\end{remark}

\begin{remark}
 The technique from \cite{CostabelDauge2015} for bounding
 $C_{\ddiv}$ extends to domains $\omega$ that are star-shaped with respect
 to all points of some open nonempty ball $B\subseteq \omega$.
\end{remark}

\begin{remark}
 The stated bounds on Korn's constant do not apply to 
 three-dimen\-sional element domains.
 Upper bounds can be numerically computed with the 
 method of \cite{Gallistl2021}, but their theoretical
 justification relies on (asymptotic) assumptions
 the verification of which turns out difficult in practice.
\end{remark}

\section{Application to the Steklov eigenvalue problem}\label{s:steklov}

Let $\Omega\subseteq\mathbb R^n$ be a domain as in prior
sections with outer unit normal $\mathbf n$.
As a model problem we consider the problem of finding
$(\lambda,w)$ with nontrivial $w$ such that
\begin{align*}
 \hspace{10ex}
 -\Delta w + w &=0 & & \hspace{-10ex}\text{in }\Omega &
\\
\hspace{10ex}
    \frac{\partial w}{\partial \mathbf{n}} &= \lambda w 
  & &\hspace{-10ex}\text{on }\partial\Omega . &\\
\end{align*}
The eigenvalue relation on the boundary subject to a
homogeneous linear partial differential equation in
the domain is related to the spectrum of a
Dirichlet-to-Neumann map.
The standard variational formulation is posed in the 
Sobolev space $H^1(\Omega)$.
Since no dual mixed formulation has been studied in the
literature so far, it is explained here in more detail than
the classical models of the foregoing sections.
The idea is to introduce the variables 
$\sigma=\nabla w$ and $(w,\gamma)$ with
$\gamma=w|_{\partial\Omega}$.
Let $\Sigma:= H_\Gamma(\ddiv,\Omega)$ be the subspace
of all $\tau\in H(\ddiv,\Omega)$ whose normal trace
$\tau\cdot\mathbf n|_{\partial\Omega}$ belongs to
$L^2(\partial\Omega)$, equipped with the norm
$$
 \|\tau\|_{H_\Gamma(\ddiv,\Omega)}
 :=
\sqrt{ \|\tau\|_{H(\ddiv,\Omega)}^2 
    + \|\tau\cdot\mathbf n\|_{L^2(\partial\Omega)}^2}
$$
and let 
$U:= L^2(\Omega)\times L^2(\partial\Omega)$.
Let $a(\cdot,\cdot):=(\cdot,\cdot)_{L^2(\Omega)}$
be chosen as the $L^2$ product
and let
$$
 b(\tau,(v,\eta)) :=
 (\ddiv\tau,v)_{L^2(\Omega)} 
  - (\tau\cdot \mathbf n,\eta)_{L^2(\partial\Omega)}.
$$
With $c(\cdot,\cdot):=(\cdot,\cdot)_{L^2(\Omega)}$
and 
$\ell(\cdot,\cdot):=(\cdot,\cdot)_{L^2(\partial\Omega)}$,
the Steklov eigenvalue problem can then be rewritten 
as system~\eqref{e:evp}.
For convenience, the eigenvalue problem is explicitly
rewritten in the following:
Seek $(\sigma, (w,\gamma))\in\Sigma\times U$ 
and
$\lambda\in\mathbb R$
such that
\begin{align*}
\begin{array}{lclclcl}
(\sigma,\tau)_{L^2(\Omega)} 
&+& (\ddiv\tau,w)_{L^2(\Omega)} 
&-& (\tau\cdot\mathbf n,\gamma)_{L^2(\partial\Omega)}
&=&0
\\
(\ddiv\sigma,v)_{L^2(\Omega)}
&-&(w,v)_{L^2(\Omega)}
&&
&=& 0
\\
-(\sigma\cdot\mathbf n,\eta)_{L^2(\partial\Omega)}
&&
&&
&=&
-\lambda (\gamma,\eta)_{L^2(\partial\Omega)}
\end{array}
\end{align*}
for all $(\tau,(v,\eta))\in\Sigma\times U$.
It is not difficult to see that the system is inf-sup
stable and the finite eigenvalues coincide with those
of the original system.
Furthermore, the relations $\ddiv\sigma=w$ and
$\sigma\cdot\mathbf n =\lambda\gamma$
allow for substitutions in the first row of the
system resulting in the equivalent eigenvalue problem
$$
(\sigma,\tau)_{L^2(\Omega)} 
+
(\ddiv\sigma,\ddiv\tau)_{L^2(\Omega)} 
=
\lambda^{-1}(\sigma\cdot\mathbf n,\tau\cdot\mathbf n)_{L^2(\Omega)} 
\quad\text{for all }\tau\in\Sigma.
$$
An analogous equivalence can be used for the discretization
and results in a positive definite system matrix,
which is beneficial from a practical point of view.

The model discretization presented here is based on 
a partition $\mathcal T$ in convex polytopes,
a subspace $\Sigma_h\subseteq \Sigma$ from an inf-sup stable
pair $(\Sigma_h$, $V_h)$ for the Laplacian
(as in Section~\ref{s:lapl})
with the compatibility condition $\ddiv\Sigma_h \subseteq V_h$,
and $U_h:=V_h\times \operatorname{tr}_{\partial\Omega}\Sigma_h$
(the symbol $\operatorname{tr}_{\partial\Omega}$
refers to the normal trace).
This implies the relation
$$
(\ddiv,\operatorname{tr}_{\partial\Omega})\Sigma_h
\subseteq
U_h
$$
sufficient for Condition~\ref{cond:incl} to hold.
The verification of Condition~\ref{cond:proj} is immediate
because the orthogonal projection in $U_h$ is the product of
the orthogonal projections with respect to the components
$w,\gamma$.
The assumption that the trace variable is discretized 
with $\operatorname{tr}_{\partial\Omega}\Sigma_h$
allows for a reduction to a positive-definite system
as outlined above for the continuous setting.

\begin{example}\label{ex:RTsteklov}
On simplicial triangulations, the simplest example is the 
pairing of
the lowest-order Raviart--Thomas space $\Sigma_h$
and the product space of piecewise constants with respect
to $\mathcal T$ and the piecewise constants with respect to
the boundary faces $\mathcal F(\partial\Omega)$,
written
$$
 U_h := P_0(\mathcal T)\times P_0(\mathcal F(\partial \Omega)).
$$
The inf-sup stability of the discrete system and the compatibility
condition
are then consequences of standard results on the 
Raviart--Thomas element
\cite{BoffiBrezziFortin2013}.
\end{example}

In what follows it is assumed that
$U_h$ contains the subspace
$P_0(\mathcal T)\times P_0(\mathcal F(\partial \Omega))$
of piecewise constant functions.
In order to verify Condition~\ref{cond:main},
it then suffices to determine a constant $\delta_h$ such that
$$
\| w - \Pi_{0,\mathcal F(\partial \Omega)} w\|_{L^2(\partial\Omega)}^2
\leq 
\delta_h^2
(
\|\nabla w\|_{L^2(\Omega)}^2 + \|w\|_{L^2(\Omega)}^2
)
$$
where
$\Pi_{0,\mathcal F(\partial \Omega)}$
is the $L^2$ projection onto the piecewise constants with 
respect to the boundary faces.
Note that $w\in H^1(\Omega)$ so that
$w|_{\partial\Omega}\in H^{1/2}(\partial\Omega)$.
This is achieved with the following trace inequality
for convex polytopes
in terms of the geometry of an inscribed simplex.

\begin{lemma}\label{l:trace}
 Let $\omega\subseteq \mathbb R^n$ be a convex polytope
 with a face $F$
 and let $T\subseteq \omega$ be an inscribed simplex with $F$ as
 one of its faces.
 Let $v\in H^1(\omega)$ with $\int_F v\,ds =0$.
 Then
 $$
  \|v\|_{L^2(F)} 
\leq 
\sqrt{\frac{\operatorname{meas}_{n-1}(F)}{\operatorname{meas}_{n}(T)}} h_T
\sqrt{\frac{n+2\pi}{n\pi^2}}
\|\nabla v\|_{L^2(\omega)}
$$
where the symbol
$\operatorname{meas}_{n-1}(F)$ denotes the $(n-1)$-dimensional
surface measure of $F$ and $\operatorname{meas}_{n}(T)$
denotes the volume of $T$.
\end{lemma}

\begin{proof}
Without loss of generality one may assume
$\int_T v\, dx =0$
because 
$\|v\|_{L^2(F)} \leq \|v - c\|_{L^2(F)}$
for any real constant $c$, in particular
for the volume average $c=\fint_T v \,dx$.
Let $P$ denote the vertex of $T$ opposite to $F$.
The integration-by-parts formula 
with the outer unit normal $\mathbf n_T$ to $\partial T$
implies
$$
n\int_T v^2 \,dx
+
\int_T(\bullet - P)\cdot\nabla (v^2) \,dx
=\int_{\partial T} v^2 (\bullet-P)\cdot\mathbf n_T\,ds.
$$
The vector $(x-P)$ is tangential to $\partial T$ for
almost all $x\in\partial T\setminus F$ and thus
orthogonal to $\mathbf n_T$.
If $x$ belongs to the interior of $F$, an elementary
geometric consideration (volume of a cone in $\mathbb R^n$) shows that 
$(x-P)\cdot \mathbf n_T 
  = n\operatorname{meas}_{n}(T)/\operatorname{meas}_{n-1}(F)$.
Therefore the integral on the right-hand side
equals
$n\operatorname{meas}_{n}(T)/\operatorname{meas}_{n-1}(F)
 \int_F v^2\,dx$.
This leads to the classical trace identity
\begin{align*}
\int_F v^2\,ds
&=
\frac{\operatorname{meas}_{n-1}(F)}{\operatorname{meas}_{n}(T)} \int_T v^2 \,dx
+
\frac{\operatorname{meas}_{n-1}(F)}{n\operatorname{meas}_{n}(T)} \int_T(\bullet - P)\cdot\nabla (v^2) \,dx
\\
&
\leq
\frac{\operatorname{meas}_{n-1}(F)}{\operatorname{meas}_{n}(T)} \int_T v^2 \,dx
+
\frac{2 h_T \operatorname{meas}_{n-1}(F)}{n\operatorname{meas}_{n}(T)} \int_T |v| \, |\nabla v| \,dx.
\end{align*}
Using the Young inequality with an arbitrary scaling parameter
$\alpha>0$, the second term on the right-hand side can be 
bounded as follows
$$
\frac{2 h_T \operatorname{meas}_{n-1}(F)}{n\operatorname{meas}_{n}(T)} \int_T |v| \, |\nabla v| \,dx
\leq 
\frac{h_T \operatorname{meas}_{n-1}(F)}{n\operatorname{meas}_{n}(T)}
\left(
\alpha \|v\|_{L^2(T)}^2 
 + \frac{1}{\alpha} \|\nabla v\|_{L^2(T)}^2
\right) .
$$
Together with the Poincar\'e bound 
$\|v\|_{L^2(T)} \leq (h_T/\pi) \|\nabla v\|_{L^2(T)}$
this yields
$$
\| v \|_{L^2(F)}^2
\leq 
\frac{\operatorname{meas}_{n-1}(F)}{\operatorname{meas}_{n}(T)}
\left(
\frac{h_T^2(1 + h_T \alpha /n)}{\pi^2}
 +
\frac{h_T}{n}
 \frac{1}{\alpha}\right) \|\nabla v\|_{L^2(T)}^2.
$$
For $\alpha:=\pi/h_T$ we obtain
$$
\| v \|_{L^2(F)}^2
\leq 
\frac{\operatorname{meas}_{n-1}(F)}{\operatorname{meas}_{n}(T)} h_T^2
\left(
\frac{1}{\pi^2}
+
\frac{2}{\pi n}
\right) \|\nabla v\|_{L^2(T)}^2.
$$
\end{proof}

If $m$ is the maximal possible number of faces of a polytope
$K$ of $\mathcal T$ and $\mathcal T$ is not a singleton set,
an overlap argument
and Lemma~\ref{l:trace} show that Condition~\ref{cond:main}
is satisfied with 
$$
\delta_h=
\sqrt{m-1}
\max_{\substack{K\in\mathcal T \\ F\subseteq T\subseteq K}}
\sqrt{\frac{\operatorname{meas}_{n-1}(F)}{\operatorname{meas}_{n}(T)}} h_T
\sqrt{\frac{n+2\pi}{n\pi^2}} .
$$
The notation $F\subseteq T\subseteq K$ indicates that
$F$ is a boundary face and there exists a simplex $T$ inscribed to
the polytope $K\in\mathcal T$ such that $F$ is simultaneously a face
of $T$ and $K$.

\begin{corollary}[guaranteed lower Steklov eigenvalue bound]
\label{c:steklov}
Assume the above setting for the mixed formulation of
the Steklov eigenproblem.
Let $\Sigma_h\subseteq\Sigma$, $U_h\subseteq U$ be an inf-sup
stable pair of finite-dimensional subspaces related to a
partition $\mathcal T$ (with at least two elements)
in convex polytopes (with at most $m$ faces)
with 
$(\ddiv,\operatorname{tr}_{\partial\Omega})\Sigma_h \subseteq U_h$
where $U_h$ contains the piecewise constants
$P_0(\mathcal T)\times P_0(\mathcal F(\partial\Omega))$.
Then, the $J$th eigenvalue $\lambda_J$ 
of \eqref{e:evp} and the $J$th
discrete eigenvalue $\lambda_{J,h}$ of \eqref{e:devp} satisfy
 $$
\frac{\lambda_{J,h}}
{1+
 (m-1)\displaystyle
\max_{\substack{K\in\mathcal T \\ F\subseteq T\subseteq K}}
 \frac{\operatorname{meas}_{n-1}(F)}{\operatorname{meas}_{n}(T)} h_T^2
 \frac{n+2\pi}{n\pi^2} 
 \lambda_{J,h}}
\leq  \lambda_J .
$$
\end{corollary}

\begin{remark}
 Under shape-regularity assumptions,
 in the bound of Corollary~\ref{c:steklov} the prefactor
 of $\lambda_{h,J}$ in the denominator is proportional
 to $h$.
\end{remark}

\begin{example}
In this example, the Steklov eigenvalue problem in two dimensions 
is discretized
with the lowest-order Raviart--Thomas finite element pairing
from Example~\ref{ex:RTsteklov}
with  respect to a regular triangulation.
The domain under consideration is the L-shaped domain
from Example~\ref{ex:lapl}
with the initial triangulation from Figure~\ref{f:L}.
Table~\ref{tab:steklov} displays
the discrete eigenvalue, the guaranteed lower bound
from Corollary~\ref{c:steklov},
and an upper bound computed with a first-order conforming
FEM,
on a sequence of uniformly refined meshes.
\begin{table}
 \begin{tabular}{cccc}
  $h\times\sqrt{2}$ & $\lambda_{1,h}$ & lower bound & upper bound\\
  \hline   
   $2^0$    &  0.340304 &   0.188241 &   0.344375 \\
   $2^{-1}$ &  0.341129 &   0.242816 &   0.342217 \\
   $2^{-2}$ &  0.341342 &   0.283844 &   0.341624 \\
   $2^{-3}$ &  0.341397 &   0.309994 &   0.341469 \\
   $2^{-4}$ &  0.341411 &   0.324951 &   0.341430 \\
 \end{tabular}
\caption{Results for the Steklov eigenvalues on the L-shaped domain.
\label{tab:steklov}}
\end{table}
\end{example}

\section{Conclusive remarks}\label{s:concl}
There are many more applications of Theorem~\ref{t:lowerbound}
beyond the model problems highlighted in this paper.
For example, the Stokes eigenvalues in the formulation
without symmetry constraint on the stress can be discretized
in a dual pseudostress formulation \cite{CaiWang2010},
which can be applied for the computation of guaranteed lower
eigenvalue bounds of the Stokes system, with
or without lower order terms.
Another example is the biharmonic eigenvalue problem,
where a dual mixed method has been provided by \cite{ChenHuang2020}.
A computable quantity $\delta_h$ then requires knowledge on
the fundamental frequency of the biharmonic operator with
free boundary condition over reference polyhedra.
Those can be numerically
computed as in \cite{CarstensenGallistl2014}.

\section*{Acknowledgment}
The author thanks Joscha Gedicke (U Bonn) for providing
the Arnold--Winther eigenvalues from
Table~\ref{tab:elast}.
The author is supported by the ERC trough
the Starting Grant \emph{DAFNE},
agreement ID 891734.

\bibliographystyle{abbrv}
\bibliography{evp}

\end{document}